\documentclass{IEEEtran}
\usepackage{amsmath}
\usepackage{amsthm}
\usepackage{amsfonts}
\usepackage{amssymb}
\usepackage{amsbsy}
\usepackage{amsopn}
\usepackage{amstext}
\usepackage{graphics}
\usepackage{textcomp}
\usepackage{bbm}
\usepackage{epsfig}
\usepackage{verbatim}
\usepackage{mathrsfs}
\usepackage{amsfonts}
\usepackage{color}
\usepackage{epsfig}

\def\eq{\displaystyle\stackrel\triangle=}

\newtheorem{theorem}{Theorem}

\newtheorem{corollary}{Corollary}

\newtheorem{lemma}{Lemma}

\newtheorem{problem}{Problem}
\newtheorem{proposition}{Proposition}
\newtheorem{remark}{Remark}

\input{tcilatex}

\begin{document}

\title{Identifiability Analysis of Noise Covariances for LTI Stochastic
Systems with Unknown Inputs}
\author{He Kong, Salah Sukkarieh, Travis J. Arnold, Tianshi Chen, Biqiang
Mu, and Wei Xing Zheng, \IEEEmembership{Fellow, IEEE} \thanks{\textbf{This
paper is formally accepted to and going to appear in \textit{IEEE
Transactions on Automatic Control}}. Corresponding author: He Kong.} \thanks{%
He Kong is with the Shenzhen Key Laboratory of Biomimetic Robotics and
Intelligent Systems, Department of Mechanical and Energy Engineering,
Southern University of Science and Technology (SUSTech), Shenzhen, 518055,
China; he is also affiliated with the Guangdong Provincial Key Laboratory of
Human-Augmentation and Rehabilitation Robotics in Universities, SUSTech,
Shenzhen, 518055, China (e-mail: kongh@sustech.edu.cn). Salah Sukkarieh is
with the Sydney Institute for Robotics and Intelligent Systems, The
University of Sydney, NSW 2006, Australia (e-mail:
salah.sukkarieh@sydney.edu.au). Travis J. Arnold is an independent
researcher, Madison, WI 53704, United States (e-mail:
travis.arnold17@gmail.com). Tianshi Chen is with the School of Data Science
and Shenzhen Research Institute of Big Data, The Chinese University of Hong
Kong, Shenzhen, 518172, China (e-mail: tschen@cuhk.edu.cn). Biqiang Mu is
with the Key Laboratory of Systems and Control, Institute of Systems
Science, Academy of Mathematics and Systems Science, Chinese Academy of
Sciences, Beijing 100190, China (e-mail: bqmu@amss.ac.cn). Wei Xing Zheng is
with the School of Computer, Data and Mathematical Sciences, Western Sydney
University, Sydney, NSW 2751, Australia (e-mail:
w.zheng@westernsydney.edu.au).}}
\maketitle

\begin{abstract}
Most existing works on optimal filtering of linear time-invariant (LTI)
stochastic systems with arbitrary unknown inputs assume perfect knowledge of
the covariances of the noises in the filter design. This is impractical and
raises the question of whether and under what conditions one can identify
the process and measurement noise covariances (denoted as $Q$ and $R$,
respectively) of systems with unknown inputs. This paper considers the
identifiability of $Q$/$R$ using the correlation-based measurement
difference approach. More specifically, we establish (i) necessary
conditions under which $Q$ and $R$ can be uniquely jointly identified; (ii)
necessary and sufficient conditions under which $Q$ can be uniquely
identified, when $R$ is known; (iii) necessary conditions under which $R$
can be uniquely identified, when $Q$ is known. It will also be shown that
for achieving the results mentioned above, the measurement difference
approach requires some decoupling conditions for constructing a stationary
time series, which are proved to be sufficient for the well-known strong
detectability requirements established by Hautus.
\end{abstract}



\section{Introduction}

Estimation under unknown inputs (whose models or statistical properties are
not assumed to be available), also called unknown input decoupled
estimation, has received much attention in the past. In the existing
literature, many uncertain phenomena in control systems have been modeled as
unknown inputs, including system faults/attacks \cite{Johansen2014}-\cite%
{Yang2019}, abrupt/impulsive disturbances or parameters \cite{Ohlsson2012}-%
\cite{Chai2017}, arbitrary vehicle tires/ground interactions \cite%
{Johansen2007}, etc. A seminal work on unknown input decoupled estimation is
due to Hautus \cite{Hautus1983} where it has been shown that the strong
detectability criterion, including a rank matching condition and the system
being minimum phase requirement, is necessary and sufficient for the
existence of a stable observer for estimating the state/unknown input for
deterministic systems\footnote{%
The strong$^{\ast }$ detectability concept was also introduced in \cite%
{Hautus1983}. The two criteria, as discussed in \cite{Hautus1983}, are
equivalent for discrete-time systems, but differ for continuous systems.}.

Works on the filtering case, e.g., \cite{Darouach2003}-\cite{Su2015}, have
similar rank matching and system being minimum phase requirements as in \cite%
{Hautus1983}. Extensions to cases with rank-deficient shaping matrices have
been discussed in \cite{Hsieh2009}-\cite{Frazzoli2016}. It has also been
shown in the above works that for unbiased and minimum variance estimation
of the state/unknown input, the initial guess of the state must be unbiased.
Very recently, connections between the above-mentioned results and Kalman
filtering (KF) of systems within which the unknown input is taken to be a
white noise of unbounded variance, have been established in \cite%
{Bitmead2019}. There are also some works dedicated to alleviating the strong
detectability conditions and the unbiased initialization requirement (see
\cite{Kong2019Auto}-\cite{Kong2020Auto} and the references therein).

However, most existing filtering works mentioned above assume that the
process and measurement noise covariances (denoted as $Q$ and $R$,
respectively) are perfectly known for the optimal filter design. This raises
the question of whether and under what conditions one can identify $Q/R$
from real data. We believe that addressing the identifiability issue of
noise covariances under arbitrary unknown inputs is important because in
practice the noise covariances are not known \textit{a priori} and have to
be identified from real closed-loop data where there might be unknown system
uncertainties such as faults, etc. Another relevant application is path
planning of sensing robots for tracking targets whose motions might be
subject to abrupt disturbances (in the form of unknown inputs), as
considered in our recent work \cite{Kong2021}.

To the best of our knowledge, \cite{Yu2016}-\cite{Pan2013} are the only
existing works on identification of stochastic systems under unknown inputs.
However, in the former two works, the unknown inputs are assumed to be a
wide-sense stationary process with rational power spectral density or
deterministic but unknown signals, respectively. Here, we do not make such
assumptions. Also, we are mainly interested to investigate the
identifiability of the original noise covariances for linear time-invariant
(LTI) stochastic systems with unknown inputs. This is in contrast to the
work in \cite{Yu2016} where the measurement noise covariance of the
considered system is assumed to be known, and the input autocorrelations are
identified from the output data and then used for input realization and
filter design. Our work is also different from subspace identification where
the stochastic parameters of the system are estimated, which can be used to
calculate the optimal estimator gain \cite{Gevers2006}. It should be
remarked that apart from filter design, knowledge of noise covariances can
also be used for other purposes such as performance monitoring \cite%
{Moni2011}.

We note that noise covariance estimation is a topic of lasting interest for
the systems and control community, and the literature is fairly mature.
Existing noise covariance estimation methods can be classified as Bayesian,
maximum likelihood, covariance matching, and correlation techniques, etc.,
(see \cite{Hu2011}-\cite{Arnold2020} and the references therein).
Especially, the correlation methods can be classified into two groups where
the state/measurement prediction error (or measurement difference), as a
stochastic time series, is computed either explicitly via a stable filter
(for example, the autocovariance least-squares (ALS) framework in \cite%
{Rawlings2006}-\cite{Rawlings2009}) or implicitly by manipulating the
measurements (see \cite{Xia2014} for the case using one step measurement,
and \cite{Dunik2018}-\cite{Moghe2019} for the case using multi-step
measurements, respectively, in computing the measurement differences).

Still, most above-mentioned noise covariance estimation methods have not
considered the case with unknown inputs. This observation motivates us to
study the identifiability of $Q/R$ for systems under unknown inputs.
Especially, we adopt the correlation-based methodology, and mainly discuss
the implicit correlation-based frameworks, in particular, the measurement
difference approach using single-step measurement.

Moreover, given that this paper focuses on the identifiability of $Q$/$R$
via the measurement difference approach using single-step measurement, some
of the assumptions (e.g., the output matrix $C$ is assumed to be of full
column rank) seem to be stringent. Nevertheless, we believe the
consideration of the case using single-step measurement serves as the first
crucial step to fully understand the identifiability of $Q$/$R$ under the
presence of unknown inputs. A thorough investigation of the more general
case using multi-step measurements is the subject of our current and future
work.

Finally, we remark that the considered problem is inherently a theoretical
one, although we are motivated by its potential applications in practice.
However, we believe that addressing the considered question specifically for
LTI systems is the first step towards a more thorough understanding on the
topic.

The remainder of the paper is structured as follows. In Section II, we
recall preliminaries on estimation of systems with unknown inputs. Section
III contains our major results for the single-step measurement case. Section
IV illustrates the theoretical results with some numerical examples. Section
V concludes the paper.

\textbf{Notation:} $A^{\mathrm{T}}$ denotes the transpose of matrix $A$. $%
\mathbf{R}^{n}$ stands for the $n$-dimensional Euclidean space. $I_{n}$
stands for the identity matrix of dimension $n$. $\mathbf{0}$ stands for the
zero matrices with compatible dimensions. $\mathbb{C}$ and $\left\vert
z\right\vert $ denote the field of complex numbers, and the absolute value
of a given complex number $z$, respectively.

\section{\label{lyap}Preliminaries and Problem Statement}

We consider the discrete-time LTI model of the plant:
\begin{subequations}
\begin{align}
x_{k+1}=Ax_{k}+Bd_{k}+Gw_{k} \\
y_{k}=Cx_{k}+Dd_{k}+v_{k}\text{ \ }
\end{align}
\label{plant}
\end{subequations}
where $x_{k}\in \mathbf{R}^{n},$ $d_{k}\in \mathbf{R}^{q}$, and
$y_{k}\in \mathbf{R}^{p}$ are the state, the unknown input, and the
output, respectively; $w_{k}\in \mathbf{R}^{g}$ and $v_{k}\in
\mathbf{R}^{p}$ represent zero-mean mutually uncorrelated process
and measurement noises with covariances $Q\in \mathbf{R}^{g\times
g}$ and $R\in \mathbf{R}^{p\times p}$, respectively; $A,B,G,C,$ and
$D$ are real and known matrices with appropriate dimensions; the
pair $(A,C)$ is assumed to be detectable.
Without loss of generality, we assume $n\geq g$ and $G\in \mathbf{R}%
^{n\times g}$ to be of full column rank (when this is not the case, one can
remodel the system to obtain a full rank shaping matrix $\overline{G}$).

For system (\ref{plant}), a major question of interest is the existence
condition of an observer/filter that can estimate the state/unknown input
with asymptotically stable error, using only the output. To address these
questions, concepts such as strong detectability and strong estimator have
been rigorously discussed in \cite{Hautus1983} for deterministic systems%
\footnote{%
Extensions of the strong detectability to linear stochastic systems have
been discussed in \cite{Frazzoli2016}.}. As remarked in \cite{Hautus1983},
the term \textquotedblleft strong\textquotedblright\ is to emphasize that
state estimation has to be obtained without knowing the unknown input. For
later use, we include the strong detectability conditions in the sequel.
Note, however, that the measurement-difference approach does not require
strong detectability since we do not need to design a filter to explicitly
estimate the state/unknown input. Instead, we manipulate the system outputs
to implicitly estimate the state and construct a stationary time series. The
required conditions associated with the measurement-difference approach are
different from the strong detectability conditions and presented in
Proposition \ref{case_abc} and Theorem \ref{decouple_sufficent}.

\begin{lemma}
\label{condition}(\cite{Hautus1983}) The following statements hold
true: (i) the system (\ref{plant}) has a strong estimator if and
only if it is strongly detectable; (ii) system (\ref{plant}) is
strongly detectable if and only if
\begin{equation}
rank\left( \left[
\begin{array}{cc}
CB & D \\
D & 0%
\end{array}%
\right] \right) =rank(D)+rank\left( \left[
\begin{array}{c}
B \\
D%
\end{array}%
\right] \right) ,  \label{rankmatching}
\end{equation}%
and all its invariant zeros are stable, i.e.,%
\begin{equation}
rank\left( \underset{\mathcal{M}\left( z\right) }{\underbrace{\left[
\begin{array}{cc}
zI_{n}-A & -B \\
C & D%
\end{array}%
\right] }}\right) =n+rank\left( \left[
\begin{array}{c}
B \\
D%
\end{array}%
\right] \right) ,  \label{miniphase}
\end{equation}%
for all $z\in \mathbb{C}\ $and $\left\vert z\right\vert \geq 1$.
\end{lemma}

Conditions (\ref{rankmatching})-(\ref{miniphase}) are the so-called rank
matching and minimum phase requirements, respectively. Note that Lemma \ref%
{condition} holds for both the deterministic and stochastic cases (hence we
use \textquotedblleft estimator\textquotedblright\ instead of KF/Luenberger
observer; for more detailed discussions on the design and stability of KF
under unknown inputs, we refer the reader to \cite{Darouach2003}-\cite%
{Frazzoli2016} and the references therein. For system (\ref{plant}), the
noise covariances $Q$ and $R$ are usually not available, and have to be
identified from data. However, all existing filtering methods for systems
with unknown inputs in the literature adopt the assumption of knowing $Q$
and $R$ exactly, which is not practical. The identificability questions of $Q
$ and/or $R$ considered in this paper are formally stated as follows:

\begin{problem}
\label{problem1}Given system (\ref{plant}) with unknown inputs, and known $%
A,B,G,C,$ and $D$, we aim to investigate the following questions: using the
measurement difference approach, whether and under what conditions one can
(i) uniquely jointly identify $Q$ and $R$; (ii) uniquely identify $Q$ or $R$%
, assuming the other covariance to be known.
\end{problem}

\section{\label{internal}Identifiability of Q/R Using the Single-step
Measurement Difference Approach}

This section contains the first major results of the paper. We show that, in
theory, the single-step measurement difference approach does not have a
unique solution for jointly estimating $Q$ and $R$ of system (\ref{plant}).
Estimating $Q$ or $R$, assuming the other to be known, will also be
considered. For deriving the results in this section, we will assume that $C$
is of full column rank. We remark that although the assumption on $C$ is
restrictive, the discussions in the sequel bring some insights into the
identifiability study of $Q$/$R$, i.e., even with the above stringent
assumption, it will be shown that only under restrictive conditions, $Q$/$R$
can be uniquely identified.


\subsection{Conditions for obtaining an unknown input decoupled stationary
time series}

When $C$ is of full column rank, from (\ref{plant}), it can be obtained that%
\begin{equation}
\begin{array}{l}
y_{k+1}=Cx_{k+1}+Dd_{k+1}+v_{k+1} \\
\text{ \ \ \ \ \ }=CAx_{k}+CBd_{k}+CGw_{k}+Dd_{k+1}+v_{k+1},%
\end{array}
\label{output}
\end{equation}%
and%
\begin{equation}
x_{k}=My_{k}-MDd_{k}-Mv_{k},  \label{state}
\end{equation}%
where%
\begin{equation}
M=(C^{\mathrm{T}}C)^{-1}C^{\mathrm{T}}.  \label{M}
\end{equation}%
By substituting (\ref{state}) into (\ref{output}), one has that
\begin{equation*}
\begin{array}{l}
\overline{z}_{k}=y_{k+1}-CAMy_{k}=(CB-CAMD)d_{k} \\
\text{ \ \ \ \ \ \ }+Dd_{k+1}+CGw_{k}+v_{k+1}-CAMv_{k}.%
\end{array}%
\end{equation*}%
Given that we do not assume to have any knowledge of the unknown input, it
is not possible for us to conduct any analysis of the statistical properties
of $\overline{z}_{k}$. Hence, a necessary and sufficient condition to
decouple the influence of the unknown input on $\overline{z}_{k}$ is the
existence of a nonzero matrix $K\in \mathbf{R}^{r\times p}$ such that%
\begin{equation}
\begin{array}{l}
z_{k}=K\overline{z}_{k}=K(CB-CAMD)d_{k}+KDd_{k+1} \\
\text{ \ \ \ \ \ \ }+KCGw_{k}+Kv_{k+1}-KCAMv_{k},%
\end{array}
\label{measurement_diff}
\end{equation}%
with%
\begin{equation}
K\underset{H\in \mathbf{R}^{p\times 2q}}{\underbrace{\left[
\begin{array}{ll}
C(B-AMD) & D%
\end{array}%
\right] }}=0.  \label{necessary_con1}
\end{equation}

\begin{remark}
\label{rem1} For the single-step measurement difference approach, later we
will establish (i) necessary conditions under which $Q$ and $R$ can be
uniquely jointly identified (see Proposition \ref{ALS_fullrank}); (ii)
necessary and sufficient conditions under which $Q$ can be uniquely
identified, when $R$ is known (see Proposition \ref{only_Q} and Corollary %
\ref{corco_estimateQ}); (iii) necessary conditions under which $R$ can be
uniquely identified, when $Q$ is known (see Proposition \ref{corco_estimateR}%
). Moreover, it will be shown that for achieving the results mentioned
above, the measurement difference approach requires some decoupling
conditions for constructing a stationary time series (see Proposition \ref%
{case_abc}). The latter conditions are proved to be sufficient (see Theorem %
\ref{decouple_sufficent}) for the strong detectability requirement in \cite%
{Hautus1983}. Also, if the existence conditions on $K$ are satisfied, then
one can use standard techniques to calculate $K$ \cite[Chap. 6]{Laub2005}.
\end{remark}

There are a few potential scenarios when (\ref{necessary_con1}) holds:%
\begin{equation}
\begin{array}{l}
\text{(a) }C(B-AMD)\neq 0\text{, }D\neq 0\text{,} \\
\text{(b) }C(B-AMD)=CB\neq 0\text{, }D=0\text{,} \\
\text{(c) }C(B-AMD)=0\text{, }D\neq 0\text{,} \\
\text{(d) }C(B-AMD)=CB=0\text{, }D=0\text{.}%
\end{array}
\label{four_cases}
\end{equation}%
Note that since $C$ is assumed to be of full column rank, case (d) in (\ref%
{four_cases}) cannot happen. This is because when $C$ is of full column rank,%
\begin{equation}
\begin{array}{l}
C(B-AMD)=CB=0\text{, }D=0 \\
\Rightarrow CB=0,\text{ }D=0, \\
\Rightarrow B=0,\text{ }D=0,%
\end{array}
\label{unknown_vanish}
\end{equation}%
i.e., the unknown input $d_{k}$ vanishes in system (\ref{plant}). Note that
here in this work we focus on the case with unknown inputs, i.e., the
situation of $B=0$ and $D=0$ is not applicable. For cases (a)-(c), we have
the following immediate results.

\begin{proposition}
\label{case_abc}Given system (\ref{plant}) with $C$ being of full column
rank, then the following statements hold true:

(i) for case (a) in (\ref{four_cases}), there exists a matrix $K$ such that
the equality in (\ref{necessary_con1}) holds if and only if%
\begin{equation}
rank(H)=2q,  \label{nece1}
\end{equation}%
where, $H$ is defined in (\ref{necessary_con1}); for condition (\ref{nece1})
to hold, it is necessary that $rank(B-AMD)=q$, $n\geq q$, $rank(D)=q$, $%
p\geq 2q$;

(ii) for case (b) in (\ref{four_cases}), there exists a matrix $K$ such that
the equality in (\ref{necessary_con1}) holds if and only if $rank(B)=q$;

(iii) for case (c) in (\ref{four_cases}), there exists a matrix $K$ such
that the equality in (\ref{necessary_con1}) holds if and only if $B-AMD=0$, $%
rank(D)=q$.
\end{proposition}

\begin{proof}
(i) For case (a), from the solution properties of matrix equations \cite[%
Chap. 6]{Laub2005}, there exists $K$ such that equalities in (\ref%
{necessary_con1}) hold if and only if $rank(H)=2q$. The rest of the proof
for part (i) follows naturally from condition (\ref{nece1}). Parts (ii) and
(iii) follow similarly.
\end{proof}

Note that case (c) is unrealistic as it requires $B-AMD=0$. However, we
include the discussion on it just for completeness. One would wonder how
stringent the decoupling condition in (\ref{necessary_con1}) and possible
cases (a)-(c) in (\ref{four_cases}) are, compared to the strong
detectability conditions in Lemma \ref{condition}. This question is answered
in the following theorem.

\begin{theorem}
\label{decouple_sufficent}For cases (a)-(c), $C$ being of full column rank
and the decoupling condition (\ref{necessary_con1}) are sufficient for the
strong detectability conditions in Lemma \ref{condition}.
\end{theorem}

\begin{proof}
We prove the claim for cases (a)-(c), respectively.

For case (a), we note from Proposition \ref{case_abc} that $D$ has to be of
full column rank. This further implies that the rank matching condition (\ref%
{rankmatching}) holds. Note that%
\begin{eqnarray*}
\begin{bmatrix}
I_{n} & AM \\
0 & I_{p}%
\end{bmatrix}%
\mathcal{M}\left( z\right) &=&%
\begin{bmatrix}
zI_{n}-A+AMC & AMD-B \\
C & D%
\end{bmatrix}
\\
&=&\underset{\overline{\mathcal{M}}\left( z\right) }{\underbrace{%
\begin{bmatrix}
zI_{n} & AMD-B \\
C & D%
\end{bmatrix}%
}}
\end{eqnarray*}%
where $\mathcal{M}\left( z\right) $ is defined in (\ref{miniphase}). 
When $D$ is of full column rank, there always exists a matrix $X\in \mathbf{R%
}^{n\times p}$ such that  $XD=AMD-B.$ Denote%
\begin{equation*}
X\left( z\right) =\left[
\begin{array}{cc}
\frac{1}{z}I_{n} & -\frac{1}{z}X \\
0 & I_{p}%
\end{array}%
\right] ,
\end{equation*}%
which is of full column rank for all $z\in \mathbb{C}\ $and $\left\vert
z\right\vert \geq 1$. Multiplying $X\left( z\right) $ on the left hand side
of $\overline{\mathcal{M}}\left( z\right) $ gives us%
\begin{equation*}
\begin{array}{l}
rank(\mathcal{M}\left( z\right) )=rank(\overline{\mathcal{M}}\left( z\right)
)=rank(X\left(z\right) \overline{\mathcal{M}}\left( z\right) ) \\
\Rightarrow rank(\overline{\mathcal{M}}\left( z\right) )=rank\left( \left[
\begin{array}{cc}
I_{n}-\frac1zXC & 0 \\
C & D%
\end{array}%
\right] \right) =n+q,%
\end{array}%
\end{equation*}%
for all $z\in \mathbb{C}\ $and $\left\vert z\right\vert \geq 1$. In other
words, the minimum phase condition in (\ref{miniphase}) holds. The proof for
case (a) is completed.

For case (b), given that $C$ and $B$ are of full column rank, and $D=0$, it
can be easily checked that conditions (\ref{rankmatching})-(\ref{miniphase})
hold.

For case (c), given that both $C$ and $D$ are of full column rank, and $B=AMD
$, it can be checked that condition (\ref{rankmatching}) holds. The matrix $%
\overline{\mathcal{M}}\left( z\right)$ appeared for the case (a) satisfies
\begin{align*}
\overline{\mathcal{M}}\left( z\right) =%
\begin{bmatrix}
zI_{n} & AMD-B \\
C & D%
\end{bmatrix}
=%
\begin{bmatrix}
zI_{n} & 0 \\
C & D%
\end{bmatrix}%
.
\end{align*}
Thus $rank(\mathcal{M}\left( z\right))=rank(\overline{\mathcal{M}}\left(
z\right))=n+q,$ for all $z\in \mathbb{C}\ $and $z\neq 0$ and hence condition
(\ref{miniphase}) holds. This completes the proof.
\end{proof}

Theorem \ref{decouple_sufficent} reveals that the measurement difference
approach requires more stringent conditions than strong detectability
conditions. As it will be discussed in the sequel, even with the above
stringent requirement, $Q$/$R$ could be potentially uniquely identified
under restrictive conditions.

\subsection{Joint identifiability analysis of $Q$ and $R$}

We next discuss the joint identifiability of $Q\ $and $R$. As such, assume
that one of the cases (a)-(c) happens so that the decoupling condition in (%
\ref{necessary_con1}) holds. From (\ref{measurement_diff}), one has%
\begin{equation}
z_{k}=KCGw_{k}+Kv_{k+1}-KCAMv_{k},  \label{final_time}
\end{equation}%
which is a zero-mean stationary time series. We also have
\begin{subequations}
\begin{align}
&S_0\eq E(z_{k}(z_{k})^{\mathrm{T}})=KCGQ(KCG)^{\mathrm{T}}+KRK^{\mathrm{T}}
\notag  \label{autocov1} \\
&\hspace{30mm}+KCAMR(KCAM)^{\mathrm{T}}, \\
&S_1\eq E(z_{k+1}(z_{k})^{\mathrm{T}})=-KCAMRK^{\mathrm{T}} ,
\label{autocov2} \\
&\hspace{8mm}E(z_{k+j}(z_{k})^{\mathrm{T}})=0,~j\geq 2,
\end{align}
where $E(\cdot)$ denotes the mathematical expectation. The above equations
give
\end{subequations}
\begin{equation}
\begin{array}{l}
S=%
\begin{bmatrix}
S_{0} \\
S_{1}%
\end{bmatrix}%
=E\left( \left[
\begin{array}{c}
z_{k}(z_{k})^{\mathrm{T}} \\
z_{k+1}(z_{k})^{\mathrm{T}}%
\end{array}%
\right] \right) \\
=\left[
\begin{array}{c}
KCG \\
0%
\end{array}%
\right] Q(KCG)^{\mathrm{T}}+\left[
\begin{array}{c}
K \\
-KCAM%
\end{array}%
\right] RK^{\mathrm{T}} \\
+\left[
\begin{array}{c}
KCAM \\
0%
\end{array}%
\right] R(KCAM)^{\mathrm{T}}.%
\end{array}
\label{autocovariance}
\end{equation}%
Denote the vectorization operator of a matrix $A=[a_1,a_2,\cdots,a_n]$ by $%
\mathrm{vec}(A)=[a_1^{\mathrm{T}},a_2^{\mathrm{T}},\cdots,a_n^{\mathrm{T}}]^{%
\mathrm{T}}$ and the Kronecker product of $A$ and $B$ by $A\otimes B$,
respectively. By applying the identity $\mathrm{vec}~\!(ABC)=(C^{\mathrm{T}%
}\otimes A)\mathrm{vec}~\!(B)$ involving the vectorization operator (\ref%
{autocovariance}) and the Kronecker product, we have the following system of
linear equations
\begin{align}  \label{LSQ1}
\mathcal{A}~\!\mathrm{vec}{([Q,R])} = \mathrm{vec}{(S)},
\end{align}
where 
\begin{equation}
\mathcal{A}=\left[
\begin{array}{cc}
\overline{K} & 0 \\
0 & \overline{K}%
\end{array}%
\right] \left[
\begin{array}{cc}
\mathcal{A}_{1} & \mathcal{A}_{2} \\
0 & -I_{p}\otimes CAM%
\end{array}%
\right] ,  \label{cap_A}
\end{equation}%
in which%
\begin{equation*}
\begin{array}{l}
\overline{K}=K\otimes K,\text{ }\mathcal{A}_{1}=CG\otimes CG\in \mathbf{R}%
^{p^{2}\times g^{2}}, \\
\mathcal{A}_{2}=(I_{p}\otimes I_{p})+(CAM\otimes CAM)\in \mathbf{R}%
^{p^{2}\times p^{2}}.%
\end{array}%
\end{equation*}%
Given the process is ergodic, a valid procedure of approximating the
expectation from the data is to use the time average. Especially, given all
the collected data as $z_{0:N}$, one has%
\begin{equation}
\widetilde{S}=[\widetilde{S}_{0},\widetilde{S}_{1}],  \label{Eab}
\end{equation}%
where%
\begin{equation}
\widetilde{S}_{0}=\frac{1}{N+1}\sum\limits_{i=0}^{N}z_{k}z_{k}^{\mathrm{T}},%
\text{ }\widetilde{S}_{1}=\frac{1}{N}\sum\limits_{i=0}^{N-1}z_{k+1}(z_{k})^{%
\mathrm{T}}.  \label{appro}
\end{equation}%
Denote%
\begin{equation*}
e=\mathrm{vec}(\widetilde{S}).
\end{equation*}%
We then have the following standard least-squares problem for identifying $%
Q\ $and $R$:%
\begin{equation}
\Xi ^{\ast }=\arg \min\limits_{\Xi }\left\Vert \mathcal{A}\Xi -e\right\Vert
^{2}  \label{LSQ}
\end{equation}%
where $\Xi =\left[ {\mathrm{vec}(\widehat{Q}),\mathrm{vec}(}\widehat{{R}}{)}%
\right] $, $e$ is defined above (\ref{LSQ}). The joint identifiability of $Q$
and $R$ is determined by the full column rankness of $\mathcal{A}$.

It should be noted that in the least-square problems listed in the remainder
of the paper, some permutation matrices can be introduced to identify the
unique elements of $Q$ and $R$, and additional constraints need to be
enforced on the $Q$ and $R$ estimates (see, e.g., \cite{Rawlings2006}-\cite%
{Rawlings2009}, \cite{Dunik2018}), given that they are both symmetric and
positive semidefinite matrices. Then the constrained least-squares problems
can be transformed to semidefinite programs \cite[Chap. 3.4]{Arnold2020} and
solved efficiently using existing software packages such as CVX \cite%
{Boyd2013}. For simplicity, we have not formally included these constraints
in the least-squares problem formulations, because this will not impact the
discussions on the solution uniqueness of these least-squares problems. It
should also be noted that in the simulation examples shown in Section V,
such symmetric and positive semidefinite constraints have been enforced.
Moreover, in the least-square problems listed in the remainder of the paper,
including (\ref{LSQ}), (\ref{estimating_Q}), we consider the most general
scenario and do not assume to have any knowledge of the structure of $Q$ and
$R$ except that they are supposed to be symmetric and positive semidefinite
matrices, which we intend to identify. In practice, if one has some
knowledge of their structures, for example, if $Q$ and/or $R$ are assumed to
be partially known or they are diagonal matrices, the least-square problems
listed in this paper can be readily modified to incorporate such knowledge.
For $\mathcal{A}$ in (\ref{LSQ}), We have the following results.

\begin{proposition}
\label{ALS_fullrank}Given system (\ref{plant}) with $C$ being of full column
rank, the following statements hold true:

(i) $\mathcal{A}$ is of full column rank only if%
\begin{equation}
rank\left( \underset{K_{M}}{\underbrace{\left[
\begin{array}{c}
K \\
KCAM%
\end{array}%
\right] }}\right) =p,\text{ }rank\left( KCG\right) =g;
\label{nece_condition_uni}
\end{equation}

(ii) when $G=I_{n}$ (i.e., $g=n$) and $p=n$, $\mathcal{A}$ is of full column
rank if and only if%
\begin{equation*}
rank\left( K\right) =p\text{, }rank\left( CAM\right) =p\text{; }
\end{equation*}

(iii) when $G=I_{n}$ (i.e., $g=n$) and $p=n$, for $\mathcal{A}$ to be of
full column rank, the unknown input $d_{k}$ has to vanish from system (\ref%
{plant}), i.e., $B=0$, $D=0$.
\end{proposition}

\begin{proof}
(i) We prove the result by contradiction. Firstly, assume that the matrix $%
K_{M}$ in (\ref{nece_condition_uni}) loses rank, i.e., there exists a
nonzero vector $h$ such that $K_{M}h=0$. Set $R=hh^{\mathrm{T}}$ so that%
\begin{equation*}
\begin{array}{l}
KCAMhh^{\mathrm{T}}K^{\mathrm{T}}=0, Khh^{\mathrm{T}}K^{\mathrm{T}}=0, \\
KCAMhh^{\mathrm{T}}(KCAM)^{\mathrm{T}}=0.%
\end{array}%
\end{equation*}%
Further by selecting $Q=0$, then one has that $\mathcal{A} \mathrm{vec}%
~\!([0,hh^{\mathrm{T}}]) =0$. This means that $\mathcal{A}$ is not of full
column rank. Similarly, now assume that $rank\left( KCG\right) < g$ and
there exists a nonzero vector $e$ such that $KCGe=0$. If we set $Q=ee^{%
\mathrm{T}}$, $R=0$, then $\mathcal{A} \mathrm{vec}~\!([ee^{\mathrm{T}},0])
=0$. Hence, $\mathcal{A}$ is not of full column rank.

(ii) When $G=I_{n}$ (i.e., $g=n$) and $p=n,$ $\mathcal{A}$\ is of full
column rank if and only if
\begin{equation*}
\begin{array}{l}
rank\left( \left[
\begin{array}{cc}
\overline{K} & 0 \\
0 & \overline{K}%
\end{array}%
\right] \right) =2p^{2}\text{, } \\
rank\left( \left[
\begin{array}{cc}
\mathcal{A}_{1} & \mathcal{A}_{2} \\
0 & -I_{p}\otimes CAM%
\end{array}%
\right] \right) =2p^{2} \\
\Leftrightarrow rank\left( K\right) =p\text{, }rank\left( CAM\right) =p\text{%
,}%
\end{array}%
\end{equation*}%
given the fact that both $\mathcal{A}_{1}$ and $I_{p}\otimes CAM$ become
square matrices when $g=p=n$.

(iii) From part (ii) of the current theorem, when $G=I_{n}$ and $p=n,$ $%
\mathcal{A}$\ is of full column rank only if $K$ is of full column rank.
Based on the decoupling condition (\ref{necessary_con1}), this further
implies that case (d) in (\ref{four_cases}) happens. From the arguments
listed in (\ref{unknown_vanish}), one has that the unknown input $d_{k}$
vanishes in system (\ref{plant}). This completes the proof.
\end{proof}

Note that for the necessary conditions in (\ref{nece_condition_uni}) to
hold, one must have that $r\geq \max \{\left\lceil \frac{p}{2}\right\rceil
,g\}$, where $\left\lceil a\right\rceil $ stands for the ceiling operation
generating the least integer not less than $a$, where $a$ is a real number.
Also, from (\ref{cap_A}), one can see that for $\mathcal{A}$ to be of full
column rank, it must hold that $2r^{2}\geq p^{2}+g^{2}$. Hence, it is
necessary that%
\begin{equation}
r\geq \max \left\{ \left\lceil \frac{p}{2}\right\rceil ,\text{ }g,\text{ }%
\left\lceil \sqrt{\frac{p^{2}+g^{2}}{2}}\right\rceil \right\} .
\label{r_lowerbound}
\end{equation}
In practice, the structure of $G$ represents how the process noise affects
the system dynamics. When no such knowledge is available, $G$ is usually
chosen to be the identity matrix. From part (i) of the above proposition, it
can be seen that for the general case where $G$ is known, we have only
established necessary condition for $\mathcal{A}$ to have full column rank.
For the special case when $G=I_{n}$ and $p=n$, although necessary and
sufficient conditions are obtained in part (ii), part (iii) further reveals
that for $\mathcal{A}$ to have full column rank, the unknown input has to be
absent from the system model, i.e., it is not an applicable case. The above
findings motivate us to take a step back, and consider part (ii) of Problem %
\ref{problem1}.

\subsection{Identifiability analysis of $Q$ when $R$ is known}

In this subsection, we investigate one case of part (ii) of Problem \ref%
{problem1}, i.e., analyze the identifiability of $Q$ when $R$ is available.
When $R$ is known, the equation \eqref{autocov1} reduces to
\begin{equation}
\mathcal{A}_{Q}\mathrm{vec}(Q)=\mathrm{vec}(S_{0})-\overline{K}\mathcal{A}%
_{2}\mathrm{vec}(R),  \label{estimating_Q}
\end{equation}%
where%
\begin{equation*}
\mathcal{A}_{Q}\displaystyle=\overline{K}\mathcal{A}_{1}
\end{equation*}%
and $\overline{K},\mathcal{A}_{1},\mathcal{A}_{2}$ are defined in (\ref%
{cap_A}). Define%
\begin{equation}
e_{Q}=\mathrm{vec}(\widetilde{S}_{0})-\overline{K}\mathcal{A}_{2}\mathrm{vec}%
(R).  \label{b_k_bar}
\end{equation}%
By following a similar procedure with the previous subsection, we have the
following standard least-squares problem formulation for identifying $Q$:%
\begin{equation}
\Xi _{Q}^{\ast }=\arg \min\limits_{\Xi _{Q}}\left\Vert \mathcal{A}_{Q}\Xi
_{Q}-e_{Q}\right\Vert ^{2}  \label{ls_onlyQ}
\end{equation}%
where $\Xi _{Q}={\mathrm{vec}(\widehat{Q})}$, $e_{Q}$ is defined in (\ref%
{b_k_bar}). Thus the identifiability of $Q$ when $R$ is known is equivalent
to the matrix $\mathcal{A}_{Q}$ being of full column rank.

\begin{proposition}
\label{only_Q}Given system (\ref{plant}) with $C$ being of full column rank,
the following statements hold true:

(i) {$\mathcal{A}_{Q}$ in (\ref{estimating_Q}) is of full column rank only
if $r\geq g$; }

(ii) for case (a) in (\ref{four_cases}), $\mathcal{A}_{Q}$ in (\ref%
{estimating_Q}) is of full column rank if and only if%
\begin{equation}
rank(H)+g=rank\left( \left[
\begin{array}{ll}
H & CG%
\end{array}%
\right] \right) ,  \label{case_a_onlyQ}
\end{equation}%
where $H$ is defined in (\ref{necessary_con1});

(iii) for case (b) in (\ref{four_cases}), $\mathcal{A}_{Q}$ in (\ref%
{estimating_Q}) is of full column rank if and only if%
\begin{equation}
rank\left( B\right) +g=rank\left( \left[
\begin{array}{ll}
B & G%
\end{array}%
\right] \right) ;  \label{case_b_onlyQ}
\end{equation}

(iv) for case (c) in (\ref{four_cases}), $\mathcal{A}_{Q}$ in (\ref%
{estimating_Q}) is of full column rank if and only if $B-AMD=0$ and
\begin{equation}
rank\left( D\right) +g=rank\left( \left[
\begin{array}{ll}
D & CG%
\end{array}%
\right] \right) .  \label{case_c_onlyQ}
\end{equation}
\end{proposition}

\begin{proof}
(i) Note that $\mathcal{A}_{Q}=KCG\otimes KCG\in \mathbf{R}^{r^{2}\times
g^{2}}$. Hence, $\mathcal{A}_{Q}$ is of full column rank if and only if $%
KCG\in \mathbf{R}^{r\times g}$ is of full column rank. Hence, for $KCG$ to
be of full column rank, it is necessary that $r\geq g$.

{(ii) For case (a), the conclusion is implied by the identity
\begin{align*}
\begin{bmatrix}
I & 0 \\
-K & I%
\end{bmatrix}
\begin{bmatrix}
H & CG \\
0 & KCG%
\end{bmatrix}
=
\begin{bmatrix}
H & CG \\
-KH & 0%
\end{bmatrix}
=%
\begin{bmatrix}
H & CG \\
0 & 0%
\end{bmatrix}%
\end{align*}
and the requirement on the full column rank of $KCG$ as well as the
decoupling condition $KH=0$ in (\ref{necessary_con1}). }

(iii) This part is straightforward by using the full column rank of $C$.

(iv) This part follows by similar arguments with part (i).

The proof is completed.
\end{proof}

We also have the following corollary when $G=I_{n}$.

\begin{corollary}
\label{corco_estimateQ} Given system (\ref{plant}) with $C$ being of full
column rank, and $G=I_{n}$, the following statements hold true:

(i) {$\mathcal{A}_{Q}$ in (\ref{estimating_Q}) is of full column rank only
if $r\geq n$; }

(ii) for case (a) in (\ref{four_cases}), $\mathcal{A}_{Q}$ in (\ref%
{estimating_Q}) is of full column rank if and only if%
\begin{equation}
rank\left( \left[
\begin{array}{ll}
C & D%
\end{array}%
\right] \right) =rank(H)+n.  \label{G_identity}
\end{equation}%
where $H$ is defined in (\ref{necessary_con1});

(iii) for case (b) in (\ref{four_cases}), $\mathcal{A}_{Q}$ in (\ref%
{estimating_Q}) is not of full column rank;

(iv) for case (c) in (\ref{four_cases}), $\mathcal{A}_{Q}$ in (\ref%
{estimating_Q}) is of full column rank if and only if $B-AMD=0$ and%
\begin{equation*}
rank\left( D\right) +n=rank\left( \left[
\begin{array}{ll}
D & C%
\end{array}%
\right] \right) .
\end{equation*}
\end{corollary}




\subsection{Identifiability analysis of $R$ when $Q$ is known}

Now, we consider the other case of part (ii) of Problem \ref{problem1},
i.e., analyze the identifiability of $R$ when $Q$ is available. When $Q$ is
known, the system of equations \eqref{autocov1}-\eqref{autocov2} becomes
\begin{equation}
\mathcal{A}_{R}\mathrm{vec}(R) = \mathrm{vec}(S)-\left[
\begin{array}{c}
\overline{K}\mathcal{A}_{1}\mathrm{vec}(Q) \\
0%
\end{array}%
\right],  \label{estimating_R}
\end{equation}%
where
\begin{equation*}
\mathcal{A}_{R}=\left[
\begin{array}{c}
\overline{K}\mathcal{A}_{2} \\
-\overline{K}(I_{p}\otimes CAM)%
\end{array}%
\right]
\end{equation*}%
with $\overline{K},\mathcal{A}_{1}$ and $\mathcal{A}_{2}$ being defined in (%
\ref{cap_A}). Similarly with the previous subsection, we have the following
results.


\begin{proposition}
\label{corco_estimateR}{Given system (\ref{plant}) with $C$ being of full
column rank, $\mathcal{A}_{R}$ is of full column rank only if $rank\left(
K_{M}\right) =p$, where $K_{M}$ is defined in (\ref{nece_condition_uni}). }
\end{proposition}

\begin{proof}
The proof follows a similar procedure with that of Proposition \ref%
{ALS_fullrank}, and is omitted.
\end{proof}

For the cases (e.g., part (iii) of Corollary \ref{corco_estimateQ} or the
conditions in (\ref{case_a_onlyQ})-(\ref{case_c_onlyQ}) do not hold) when
the solutions to the systems of linear equations are not unique, a natural
idea is to use regularization to introduce further constraints to uniquely
determine the solution \cite{Chen2013}. However, a key question to be
answered is whether some desirable properties can be guaranteed for the
covariance estimates. A full investigation of the above questions is the
subject of our current and future work.

\section{\label{exten_delay} Numerical Examples}

We next use some numerical examples to illustrate the theoretical results.
Firstly, consider the plant model (\ref{plant}) with%
\begin{equation*}
\begin{array}{l}
A=\left[
\begin{array}{ll}
1 & 1 \\
0 & 1%
\end{array}%
\right] ,\text{ }B=\left[
\begin{array}{ll}
0 & 1 \\
1 & 0%
\end{array}%
\right] ,G=\left[
\begin{array}{l}
1 \\
1%
\end{array}%
\right] , \\
C=\left[
\begin{array}{ll}
1 & 0 \\
1 & 1%
\end{array}%
\right] ,\text{ }D=\left[
\begin{array}{ll}
1 & 0 \\
2 & 0%
\end{array}%
\right] .%
\end{array}%
\end{equation*}%
It can be verified%
\begin{equation*}
H=\left[
\begin{array}{llll}
-2 & 1 & 1 & 0 \\
-2 & 1 & 2 & 0%
\end{array}%
\right]
\end{equation*}%
so that the above model fits case (a) in (\ref{four_cases}). Also, from (\ref%
{r_lowerbound}), it is necessary that $r\geq 2$. Select%
\begin{equation*}
K=\left[
\begin{array}{ll}
t_{11} & t_{12} \\
t_{21} & t_{22}%
\end{array}%
\right] .
\end{equation*}%
From the decoupling condition in (\ref{necessary_con1}), we then have%
\begin{equation*}
\begin{array}{l}
\left[
\begin{array}{llll}
-2(t_{11}+t_{12}) & t_{11}+t_{12} & t_{11}+2t_{12} & 0 \\
-2(t_{21}+t_{22}) & t_{21}+t_{22} & t_{21}+2t_{22} & 0%
\end{array}%
\right] =0 \\
\Leftrightarrow t_{11}=t_{12}=t_{21}=t_{22}=0.%
\end{array}%
\end{equation*}%
Note that increasing the row dimension of $K$ still leads to the same
conclusion, i.e., $K$ is a zero matrix. In this case, we have $\mathcal{A=}$
$0$ in (\ref{LSQ}), $\mathcal{A}_{Q}=0$ in (\ref{estimating_Q}), $\mathcal{A}%
_{R}=0$ in (\ref{estimating_R}), i.e., the noise covariances $Q/R$ are
unidentifiable at all.

Secondly, consider the plant model (\ref{plant}) with%
\begin{equation*}
\begin{array}{l}
A=\left[
\begin{array}{ll}
1 & 1 \\
0 & 1%
\end{array}%
\right] ,\text{ }B=\left[
\begin{array}{l}
1 \\
0%
\end{array}%
\right] ,\text{ }G=\left[
\begin{array}{ll}
1 & 0 \\
1 & -2%
\end{array}%
\right] . \\
C=\left[
\begin{array}{cc}
1 & -2 \\
1 & 1 \\
-2 & 1%
\end{array}%
\right] ,\text{ }D=\left[
\begin{array}{l}
1 \\
2 \\
1%
\end{array}%
\right]%
\end{array}%
\end{equation*}%
It can be verified%
\begin{equation*}
H=\left[
\begin{array}{cc}
1 & 1 \\
0 & 2 \\
-1 & 1%
\end{array}%
\right]
\end{equation*}%
so that the above model fits case (a) in (\ref{four_cases}). Also, from (\ref%
{r_lowerbound}), it is necessary that $r\geq 3$. Select%
\begin{equation}
K=\left[
\begin{array}{lll}
t_{11} & t_{12} & t_{13} \\
t_{21} & t_{22} & t_{23} \\
t_{31} & t_{32} & t_{33}%
\end{array}%
\right] .  \label{K123}
\end{equation}%
From the decoupling condition in (\ref{necessary_con1}), we then have%
\begin{equation*}
\begin{array}{l}
\left[
\begin{array}{c}
\begin{array}{ll}
t_{11}-t_{13} & t_{11}+2t_{12}+t_{13} \\
t_{21}-t_{23} & t_{21}+2t_{22}+t_{23}%
\end{array}
\\
\begin{array}{ll}
t_{31}-t_{33} & t_{31}+2t_{32}+t_{33}%
\end{array}%
\end{array}%
\right] =0 \\
\Leftrightarrow \left.
\begin{array}{l}
t_{11}=t_{13},\text{ }t_{12}=-t_{11}; \\
t_{21}=t_{23},\text{ }t_{22}=-t_{21}; \\
t_{31}=t_{33},\text{ }t_{32}=-t_{31}.%
\end{array}%
\right.%
\end{array}%
\end{equation*}%
If we set%
\begin{equation*}
K=\left[
\begin{array}{lll}
1 & -1 & 1 \\
2 & -2 & 2 \\
3 & -3 & 3%
\end{array}%
\right] ,
\end{equation*}%
it can be verified that neither of the two necessary conditions in (\ref%
{nece_condition_uni}) is satisfied. In particular, $rank(K_{M})=2,$ $%
rank\left( KCG\right) =1$. Hence, it is not possible for $\mathcal{A}\in
\mathbf{R}^{18\times 13}$ in (\ref{LSQ}) to have full column rank. In fact,
it can be checked that $rank\left( \mathcal{A}\right) =2$, i.e., $Q$ and $R$
are not uniquely jointly identifiable. Moreover, it can be calculated that
for $\mathcal{A}_{Q}\in \mathbf{R}^{9\times 4}$ in (\ref{estimating_Q}), we
have $rank\left( \mathcal{A}_{Q}\right) =1$. This reinforces the results of
Corollary \ref{only_Q} since one can easily see that $rank(H)+2=4\neq
rank\left( \left[
\begin{array}{ll}
H & CG%
\end{array}%
\right] \right) =3$, i.e., the condition in (\ref{case_a_onlyQ}) does not
hold. In other words, assuming $R$ to be known, $Q$ is not uniquely
identifiable. Similarly, for $\mathcal{A}_{R}\in \mathbf{R}^{18\times 9}$ in
(\ref{estimating_R}), since $rank(K_{M})=2$, from Corollary \ref%
{corco_estimateR}, one has that $\mathcal{A}_{R}$ cannot have full column
rank. To double confirm, it can be checked that $rank\left( \mathcal{A}%
_{R}\right) =2$, i.e., $R$ is not uniquely identifiable, when $Q$ is assumed
to be known. Note that increasing the row dimension of $K$ still leads to
the same conclusions as above for this example.

Thirdly, consider the plant model (\ref{plant}) with%
\begin{equation*}
\begin{array}{l}
A=\left[
\begin{array}{ll}
1 & 1 \\
0 & 1%
\end{array}%
\right] ,\text{ }B=\left[
\begin{array}{l}
1 \\
0%
\end{array}%
\right] ,\text{ }G=\left[
\begin{array}{l}
1 \\
4%
\end{array}%
\right] , \\
C=\left[
\begin{array}{cc}
1 & -2 \\
1 & 1 \\
-2 & 1%
\end{array}%
\right] ,\text{ }D=0.%
\end{array}%
\end{equation*}%
It can be verified%
\begin{equation*}
H=\left[
\begin{array}{cc}
1 & 0 \\
1 & 0 \\
-2 & 0%
\end{array}%
\right]
\end{equation*}%
so that the above model fits case (b) in (\ref{four_cases}). From (\ref{r_lowerbound}), it is necessary that $r\geq 3$. Denote $K$ as in (\ref{K123}%
). From the decoupling condition in (\ref{necessary_con1}), we then have%
\begin{equation*}
\left[
\begin{array}{c}
\begin{array}{ll}
t_{11}+t_{12}-2t_{13} & 0 \\
t_{21}+t_{22}-2t_{23} & 0%
\end{array}
\\
\begin{array}{ll}
t_{31}+t_{32}-2t_{33} & 0%
\end{array}%
\end{array}%
\right] =0.
\end{equation*}%
Select%
\begin{equation*}
K=\left[
\begin{array}{lll}
1 & -1 & 0 \\
1 & 3 & 2 \\
2 & 4 & 3%
\end{array}%
\right] .
\end{equation*}%
One then has that $rank(K_{M})=2,$ $rank\left( KCG\right) =1$, i.e., it is
not possible for $\mathcal{A}\in \mathbf{R}^{18\times 10}$ in (\ref{LSQ}) to
have full column rank. In fact, it can be checked that $rank\left( \mathcal{A%
}\right) =5$, i.e., $Q$ and $R$ are not uniquely jointly identifiable. Also,
it can be obtained that for $\mathcal{A}_{Q}\in \mathbf{R}^{9\times 1}$ in (%
\ref{estimating_Q}), we have $rank\left( \mathcal{A}_{Q}\right) =1$. This
reinforces the results of Corollary \ref{only_Q} since it can be confirmed
that $rank\left( B\right) +1=2=rank\left( \left[
\begin{array}{ll}
B & G%
\end{array}%
\right] \right) $, i.e., the condition in (\ref{case_b_onlyQ}) holds. In
other words, assuming $R$ to be known, $Q$ is uniquely identifiable.
Similarly, for $\mathcal{A}_{R}\in \mathbf{R}^{18\times 9}$ in (\ref%
{estimating_R}), since $rank(K_{M})=2$, from Corollary \ref{corco_estimateR}%
, one can conclude that $\mathcal{A}_{R}$ cannot have full column rank. To
double confirm, it can be checked that $rank\left( \mathcal{A}_{R}\right) =4$%
, i.e., $R$ is not uniquely identifiable, when $Q$ is assumed to be known.
Note that increasing the row dimension of $K$ still leads to the same
conclusions as above for this example.

Next for the third example, assume that the true covariances are $Q=1,$ $%
R=0.1I_{3}$. With the above system information, we follow the
procedure in Section III. C, and estimate $Q$, assuming $R$ to be
known. We run the simulation for 500 scenarios. For each scenario,
we use in total 1000 data points to estimate $S_{0}$ as in
(\ref{appro}). and the estimate for $Q$ is obtained by solving the
optimization problem (\ref{ls_onlyQ}). The results are shown in
Figure \ref{Q_es} for the above-mentioned 500 different scenarios.
It can be seen from Figure \ref{Q_es} that the estimates for $Q$ are
well dispersed around its true value. We finally remark that for
solving (\ref{ls_onlyQ}), an additional positive semidefinite
constraint has been enforced on the $Q$ estimates (i.e., for this
example, $\widehat{Q}$ is a nonnegative scalar). The optimization
problem is transformed to a standard semidefinite program and solved
by cvx \cite{Boyd2013}.
\vspace{0.5cm}
\begin{figure}[htbp]
\centering
\includegraphics[width=0.335\textwidth,bb=108 221 485
505]{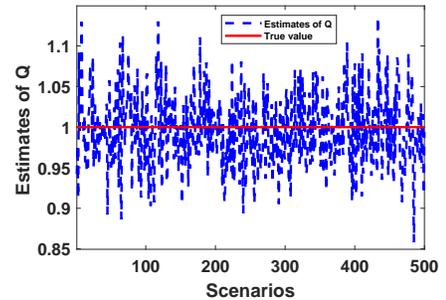} \caption{Estimates for $Q$ under 500 different
scenarios} \label{Q_es}
\end{figure}

\section{\label{conclusion}Conclusions}

The past few decades have witnessed much progress in optimal filtering for
systems with arbitrary unknown inputs and stochastic noises. However, the
existing works assume perfect knowledge of the noise covariances in the
filter design, which is impractical. In this paper, for stochastic systems
under unknown inputs, we have investigated the identifiability question of
the process and measurement noises covariances (i.e., $Q$ and $R$) using the
correlation-based measurement difference approach.

More specifically, we have focused on the single-step measurement case, and
established (i) necessary conditions under which $Q$ and $R$ can be uniquely
jointly identified (see Proposition \ref{ALS_fullrank}); (ii) necessary and
sufficient conditions under which $Q$ can be uniquely identified, when $R$
is known (see Proposition \ref{only_Q} and Corollary \ref{corco_estimateQ});
(iii) necessary conditions under which $R$ can be uniquely identified, when $%
Q$ is known (see Proposition \ref{corco_estimateR}). Moreover, it has been
shown that for achieving the results mentioned above, the measurement
difference approach requires some decoupling conditions for constructing a
stationary time series (see Proposition \ref{case_abc}). The latter
conditions are proved to be sufficient (see Theorem \ref{decouple_sufficent}%
) for the strong detectability requirement in \cite{Hautus1983}.

The above findings reveal that only under restrictive conditions, $Q$/$R$
can be potentially uniquely identified. This not only helps us to have a
better understanding of the applicability of existing filtering frameworks
under unknown inputs (since almost all of them require perfect knowledge of
the noise covariances) but also calls for further investigation of
alternative and more viable noise covariance methods under unknown inputs.

\section{Acknowledgment}

The authors would like to thank the reviewers and Editors for their
constructive suggestions which have helped to improve the quality and
presentation of this paper significantly. He Kong's work was supported by
the Science, Technology, and Innovation Commission of Shenzhen Municipality
[Grant No. ZDSYS20200811143601004]. Tianshi Chen's work was supported by the
Thousand Youth Talents Plan funded by the central government of China, the
Shenzhen Science and Technology Innovation Council under contract No.
Ji-20170189, the President's grant under contract No. PF. 01.000249 and the
Start-up grant under contract No. 2014.0003.23 funded by the Chinese
University of Hong Kong, Shenzhen.

\end{document}